\documentclass[a4paper,11pt]{article}

\usepackage[left=2.7cm,right=2.7cm,top=3.5cm,bottom=3cm]{geometry}
\usepackage{amsmath,amssymb,amsthm,amscd,amsfonts}
\usepackage{latexsym}
\usepackage{graphicx}
\usepackage[all]{xy}
\usepackage{color}
\usepackage[T1]{fontenc}
\usepackage{authblk}

\usepackage{centernot}
\usepackage{mathtools}
\usepackage{ stmaryrd }

\newcommand\ddfrac[2]{\frac{\displaystyle #1}{\displaystyle #2}}

\newcommand{\ZZ}{{\mathbb Z}}
\newcommand{\Z}{{\mathbb Z}}
\newcommand{\Q}{\mathbb Q}

\newcommand{\E}{{\mathcal{E}}}
\newcommand{\QQ}{\mathbb Q}
\newcommand{\mcF}{{\mathcal{F}}}
\newcommand{\A}{\textrm{A}}
\newcommand{\G}{\mathcal{G}}
\newcommand{\loc}{\textrm{loc}}

\newtheorem{thm}{Theorem}[section]

\newtheorem{prop}[thm]{Proposition}
\newtheorem{Problem}[thm]{Problem}

\theoremstyle{definition}
\newtheorem{definition}[thm]{Definition}
\newtheorem{remark}[thm]{Remark}

\newcommand{\z}{{\zeta}}
\newcommand{\ov}{\overline}

\newcommand{\Or}{\mathcal{O}}
\newcommand{\X}{\mathcal{X}}
\newcommand{\Hil}{\mathcal{H}}
\newcommand{\Ab}{\mathcal{A}}

\newcommand{\Bi}{\mathfrak{B}}

\DeclareMathOperator{\Gal}{\rm Gal}
\DeclareMathOperator{\Id}{\rm Id}
\DeclareMathOperator{\GL}{{\rm GL}}
\DeclareMathOperator{\SL}{{\rm SL}}

\renewcommand{\leq}{\leqslant}
\renewcommand{\geq}{\geqslant}

\title{On 5-torsion of CM elliptic curves}

\author[1]{Laura Paladino}
\affil[1]{Dipartimento di Matematica, Universit\`a della Calabria\\
Ponte Pietro Bucci, Cubo 30B - 87036 Arcavacata di Rende (CS), Italy, e-mail: paladino@mat.unical.it}

\date{ }

\begin{document}

\maketitle

Keywords: elliptic curves; complex multiplication; torsion points;

\bigskip

Mathematics subject classification: 11G05; 11F80; 11G18

\begin{abstract}
Let $\E$ be an elliptic curve defined over a number field $K$.
Let $m$ be a positive integer.
We denote by $\E[m]$ the $m$-torsion subgroup of $\E$ and
by $K_m:=K(\E[m])$ the number field obtained by adding to $K$ the coordinates
of the points of $\E[m]$. We describe the fields $K_5$, when
$\E$ is a CM elliptic curve defined over $K$, with Weiestrass form
either $y^2=x^3+bx$ or $y^2=x^3+c$. In particular we classify the fields $K_5$ in terms of generators,
degrees and Galois groups. Furthermore we show some applications of those results to the Local-Global Divisibility
Problem, to modular curves and to Shimura curves. 
\end{abstract}

\section{Introduction} \label{sec1}
Let $\E$ be an elliptic curve defined over a number field $K$ with algebraic closure $\bar{K}$. Let $m$ be a positive integer.
We denote by $\E[m]$ the $m$-torsion subgroup of $\E$ and by $K_m:=K(\E[m])$
the number field generated by the $5$-torsion points of $\E$, i.e. the field obtained by adding to $K$ the coordinates
of the points of $\E[m]$. Since $K_m$ is the splitting field of the $m$-division polynomials,
then $K_m/K$ is a Galois extension, whose Galois group we denote by $G$. For every point $P\in \E$, we indicate by $x(P)$, $y(P)$ its coordinates.
Furthermore, for every positive integer $n$, we indicate the $n$-th multiple of $P$ simply by $nP$. It is well-known that
$\E[m]\simeq (\Z/m\Z)^2$. Let $\{P_1\,,P_2\}$ be a $\Z$-basis for $\E[m]$; thus $K_m=K(x(P_1),x(P_2),y(P_1),y(P_2))$.
To ease notation, we put $x_i:=x(P_i)$ and $y_i:=y(P_i)$ ($i=1,2$).
Knowing explicit generators for $K_m$ could have a lot of interesting applications, for instance about 
Galois representations, local-global problems on elliptic curves (see \cite{Pal}, \cite{Pal2} and \cite{Pal4}),
descent problems (see for example \cite{SS} and the particular cases \cite{Ba} and \cite{Ba2}),
points on modular curves (see \cite{BP}, \cite{BP2})  and points on Shimura curves (see Subsection \ref{sec_shimura}). Anyway there are not many papers about
the argument (see also \cite{Ad}, \cite{Mer} and \cite{BZ}).  A recent and very interesting paper about number fields $\QQ(\E[m])$ is \cite{GJ}. The discussion
there is restricted to the case when $\QQ(\E[m])/\QQ$ is an abelian extension, even in the case of CM elliptic curves.
Among other results (see also Remark \ref{rem1}), in particular the authors prove that if $\E$ is an elliptic curve with complex multiplication and 
$\QQ(\E[m])/\QQ$ is abelian, then $m\in \{2,3,4\}$. 
 In this paper we will describe all possible extensions (even not abelian) $K(\E[5])/K$, for every $K$, when $\E$ is a CM elliptic curve.
We will classify them in terms of generators, degree and Galois groups.
By Artin's primitive element theorem, we know that the extension $K_m/K$ is monogeneous and one can find a single generator for $K_m/K$ by combining the above coordinates.
Anyway, in general it is not easy to find this single generator. So, during the last few years we have
searched for systems of generators easier to be found and to be used in applications.
For every $m$, by the properties of the Weil pairing $e_m$, we have that the image $z_m:=e_m(P_1,P_2)\in K_m$ is a primitive
$m$-th root of unity  and that $K(\z_m)\subseteq K_m$ (see for instance \cite{Sil}). When $m$ is odd, another generating
set for $K_m$ is showed in the following statement (see \cite{BP2}).

\begin{thm} \label{bp2,thm1}
In the notation as above, we have
\begin{equation} \label{genset} K_m=(x_1,\zeta_m,y_2), \end{equation}
for all odd integers $m$.
\end{thm}

\noindent Of course, in general it is easier to work with the generating set as
in \eqref{genset}. Furthermore, that generating set is often minimal among the
subsets of $\{x_1,x_2,\z_m,y_1,y_2 \}$ (for further details
see \cite{BP2}). For $m=3$ and $m=4$ there are explicit descriptions of all
possible number fields $K_3$ and $K_4$, in terms of generators, degrees and
Galois groups (see in particular \cite{BP2} and also \cite{BP}).
 Here we give a similar classification of every possible number fields $K_5$, for all elliptic curves with complex multiplication, belonging to the families:


$$ \hspace{1cm}  \mcF_1 : y^2=x^3+bx, \hspace{0.3cm}  \textrm{ with } b\in K  \hspace{1cm} \textrm{and} \hspace{1cm}   y^2=x^3+c, \hspace{0.3cm} \textrm{ with } c\in K.$$ \normalcolor

\bigskip \noindent 
We will treat separately the case of the family $\mcF_1: y^2=x^3+bx,$ and of the family  $\mcF_2:  y^2=x^3+c$, with $c\in K$.
 In the very last part of the paper, we show some applications (of those results) to the Local-Global Divisibility Problem, to $K$-rational CM points of modular curves
 and to $K$-rational CM points of Shimura curves. 









\section{Generators of $K(\E[5])$ for elliptic curves $y^2=x^3+bx$}  \label{sec_gen_1}
If ${\mathcal{E}}$ is an elliptic curve defined over $K$, with Weierstrass form $y^2=x^3+bx+c$,
then the abscissas of the points of order 5 of ${\mathcal{E}}$ are the roots of the polynomial

\vspace{0.7cm}
\noindent \hspace{0.3cm} $p_5(x):=-5x^{12}-62bx^{10}-380cx^9+105b^2x^8-240bcx^7+(240c^2+300b^3)x^6+696b^2cx^5+$

\bigskip\noindent \hspace{0.9cm} $(1920bc^2+125b^4)x^4+(1600c^3+80b^3c)x^3+(240b^2c^2+50b^5)x^2+(640bc^3+100b^4c)x+$

\bigskip\noindent \hspace{0.9cm} $256c^4+32b^3c^2-b^6.$

\bigskip\noindent
If $\E_1: y^2=x^3+c$ is an elliptic curve of the family  $\mcF_1$, then the abscissas of the points of order 5 of ${\mathcal{E}}$ are the roots of the polynomial

 $$q_5(x):=-5x^{12}-62bx^{10}+105b^2x^8+300b^3x^6+125b^4x^4+50b^5x^2-b^6.$$

\vspace{0.5cm}
 A factorization of $q_5(x)$ over $K(\zeta_5)$ is 

  $$q_5(x)= -5\cdot (x^4+(- 8\z_5^3  - 8\z_5^2  + 2)b x^2  + (- 8\z_5^3  - 8\z_5^2  + 5)b )\cdot (x^4 + \frac{2}{5} b x^2 + \frac{1}{5} b^2 )$$

\bigskip\hspace{1.2cm} $\cdot (x^4 + (8\z_5^3  + 8\z_5^2  + 10)b x^2  + (8\z_5^3  + 8\z_5^2  + 13)b )$

\bigskip\noindent and a factorization of $q_5(x)$ over $K(i,\zeta_5)$ is

  $$q_{5}(x)= -5 \cdot (x^2 + ((- 4i + 4)\z_5^3  + 4\z_5^2  - 4i \z_5 - 2i + 5)b) \cdot (x^2 + (-4\z_5^3  +(-4i -4)\z_5^2  - 4i \z_5 - 2i + 1)b)$$

 \bigskip\hspace{0.5cm} $\cdot (x^2 + (( 4i + 4)\z_5^3  + 4\z_5^2  + 4i \z_5 + 2i + 5)b)\cdot (x^2 + (-4\z_5^3  +(4i -4)\z_5^2  + 4i \z_5 + 2i + 1)b)$

\bigskip\hspace{0.5cm} $\cdot (x^2+$ {\Large $\frac{-2i+1}{5}$} $b  )\cdot (x^2+$ {\Large $\frac{2i+1}{5}$ } $b),$

\noindent where as usual we denote by $i$ a root of $x^2+1=0$.

\bigskip\noindent
\begin{remark} \label{rem2} Let $\phi_1$ denote the
complex multiplication of $\E_2$, i. e. $\phi_1(x,y)=(-x,iy)$. As above, in many cases if $P$ is a
nontrivial $m$-torsion point, then $\phi_1(P)$ is an $m$-torsion point that is not a multiple of $P$.
In this case a basis for $\E[m]$ is given by $\{P,\phi_1(P)\}$. 
Anyway, in a few special cases the point $ \phi_1(P)$ is a multiple of $P$. For example,
let $\omega_1:=- (1+2i)/5$, let $x_{1/2}=\pm \sqrt{\omega_1}$ and let $P_1$ and $P_2$ be the two 5-torsion points
of $\E_1$, with abscissas respectively $x_1$ and $x_2$. Since $\phi_1(P_i)=2P_i$ (for $i=1,2$), then
$\{P_i,\phi_1(P_i)\}$ is not a basis of $\E[5]$. We would have not this problem by choosing a
root of $q_5(x)$, different from $x_1$ and $x_2$.
\end{remark}

\begin{thm} \label{gen2} Let

$$\theta_1:=- ((- 4i + 4)z_5^3  + 4z_5^2  - 4i z_5 - 2i + 5).$$
\bigskip

\noindent Then $K_5=K( \zeta_5, i,\sqrt{(\theta_1+1)b\sqrt{\theta_1b}})$. \normalcolor
\end{thm}

\begin{proof}
 If $x_1:=\sqrt{\theta_1 b}$, then by the factorization of $q_5(x)$ showed above,
we have that $x_1$ is the abscissas of a $5$-torsion point of $\E$. Let $P_1=(x_1,y_1)$,
where $y_1=\sqrt{(\theta_1+1)b\sqrt{\theta_1b}}$. By calculating
$\phi_1(P_1)$ and the powers of $P_1$, one sees that
$\phi_1(P_1)$ is not a multiple of $P_1$. 
In addition observe that
$\sqrt{\theta_1 b}\in K( \zeta_5, i,\sqrt{(\theta_1+1)b\sqrt{\theta_1b}})$.
Then the conclusion follows by Remark \ref{rem2}.
\end{proof}

Observe that  $[K_5:K]\leq 2\cdot 4\cdot 2\cdot 2=32$, for every $b\in K$.
This is in accordance with the fact that $\E$ has complex multiplication
$(x,y)\mapsto (-x,-iy)$ and then the Galois representation
\[ \rho_{\E,5}:\Gal(\ov{K}/K) \rightarrow  \GL_2(\Z/5\Z) \]
is not surjective.





\section{Degrees $[K_5:K]$ for the curves of $\mcF_1$} \label{sub3}

\begin{thm}
Let $\E: y^2=x^3+bx$, with $b\in K$. Let $\theta_1$ as above. Consider the conditions

\bigskip

\begin{tabular}{lll}
& {\bf \A.} \hspace{0.3cm} $i\notin K$; \hspace{1.5cm}  & {\bf C.} \hspace{0.3cm} $\sqrt{(\theta_1)b}\notin K(i,\z_5)$; \\
& {\bf B1.} \hspace{0.3cm} $\z_5+\z_5^{-1}\notin K$;  & {\bf D.} \hspace{0.3cm} $\sqrt{(\theta_1+1)b\sqrt{\theta_1b}}\notin K(i,\z_5,\sqrt{\theta_1b})$. \\
& {\bf B2.} \hspace{0.3cm} $\z_5\notin K(\z_5+\z_5^{-1})$;  \hspace{1.5cm} &\\
\end{tabular}

\bigskip

The possible degrees of the extension $K_5/K$ are the following

\begin{center}
\begin{tabular}{|c|c|c|c|}
\hline
 $d$ & {\em holding conditions}  &  $d$ & {\em holding conditions}  \\
\hline
  {\em 32} &  \textrm{5 among } {\bf A}, {\bf B1},  {\bf B2}, {\bf C}, {\bf D} &  {\em 4} &
\textrm{2 among } {\bf A}, {\bf B1},  {\bf B2}, {\bf C}, {\bf D}\\
\hline
 {\em 16} &  \textrm{4 among } {\bf A}, {\bf B1}, {\bf B2}, {\bf C}, {\bf D} &  {\em 2} &
\textrm{1 among } {\bf A}, {\bf B1},  {\bf B2}, {\bf C}, {\bf D}\\
\hline
  {\em 8} &  \textrm{3 among } {\bf A}, {\bf B1}, {\bf B2}, {\bf C}, {\bf D} &  {\em 1} &
\textrm{ no holding conditions }  \\
\hline
\multicolumn{4}{c}{  }\\
\multicolumn{4}{c}{{\em Table 2}} \\
\end{tabular}
\end{center}

\end{thm}

\begin{proof}
Consider the tower of extensions

\[
\begin{split}  K & \subseteq K(i)\subseteq K(i,\zeta_5+\zeta_5^{-1}) \subseteq K(i,\zeta_5) \subseteq K(i,\zeta_5,\sqrt{(\theta_1)b})
\subseteq K(\zeta_3,\zeta_5, \sqrt{(\theta_1+1)b\sqrt{(\theta_1)b}}).\\
\end{split}
\]

\noindent The degree of $K_5/K$ is the product of the degrees of the intermediate extensions appearing in the tower.
Clearly each of those extensions gives a contribution to the degree less than or equal to 2.
The final computation is straightforward.
\end{proof}

\section{Galois groups $\Gal(K_5/K)$ for the curves of $\mcF_1$} \label{gal1}

Let $E_1$ be a curve of the family $\mcF_1$, let $G:=\textrm{Gal}(K(\E_2[5])/K)$ and let $d:=|G|$.
Let  $\theta_1$ and $\omega_1$  as above and let

\[
\begin{split}
\theta_2&:=- (- 4z_5^3  + (-4i-4)z_5^2  - 4i z_5 - 2i + 1);\\
\theta_3&:=- ((4i + 4)z_5^3  + 4z_5^2  +4i z_5 + 2i + 5); \\
\theta_4&:=- (- 4z_5^3  + (4i-4)z_5^2  + 4i z_5 + 2i + 1);\\
\omega_2&:=- \ddfrac{2i+1}{5}b.\\
\end{split} \]

\noindent If $P=(x,y)$ is a point of $\E$, to ease notation, let us denote by $iP$ the point $\phi_1(P)=(-x,iy)$. 
 The 24 points of exact order 5 of $\E_2$ are the following:

\[
\begin{split}
 \pm P_1&:=(x_1,\pm y_1)=\left(\sqrt{\theta_1b}, \pm \sqrt{(\theta_1+1)b\sqrt{\theta_1b}}\right) \hspace{1cm} \pm iP_1:=(-x_1,\pm iy_1);\\
 \pm P_2&:=(x_2,\pm y_2)=\left(\sqrt{\theta_2b}, \pm \sqrt{(\theta_2+1)b\sqrt{\theta_2b}}\right) \hspace{1cm} \pm iP_2:=(-x_2,\pm iy_2);\\
\pm P_3&:=(x_3,\pm y_3)=\left(\sqrt{\theta_3b}, \pm \sqrt{(\theta_3+1)b\sqrt{\theta_3b}}\right) \hspace{1cm} \pm iP_3:=(-x_3,\pm iy_3);\\
 \pm P_4&:=(x_4,\pm y_4)=\left(\sqrt{\theta_4b}, \pm \sqrt{(\theta_4+1)b\sqrt{\theta_4b}}\right) \hspace{1cm} \pm iP_4:=(-x_4,\pm iy_4);\\
\pm P_5&:=(x_5,\pm y_5)=\left(\sqrt{\omega_1b}, \pm \sqrt{(\omega_1+1)b\sqrt{\omega_1b}}\right) \hspace{0.8cm} \pm iP_5:=(-x_5,\pm iy_5);\\
\pm P_6&:=(x_6,\pm y_6)=\left(\sqrt{\omega_2b}, \pm \sqrt{(\omega_2+1)b\sqrt{\omega_2b}}\right) \hspace{0.8cm} \pm iP_6:=(-x_6,\pm iy_6).
\end{split} \]

\noindent By the observations made in the previous sections about the generators of $K_5$ and about the degree $[K_5:K]$, we have that
The Galois group is generated by the following 3 automorphisms.

\begin{description}

  \item[$i)$] The automorphism $\phi_1$ of order 4 given by the complex multiplication. We have $\phi_1(x,y)=(-x,iy)$, for all $(x,y)\in K(\E[5])$. 
 In particular, for every $1\leq j\leq 6$, the automorphism $\phi_1$ maps $\sqrt{\theta_jb}$ to 
$-\sqrt{\theta_jb}$ (i. e. $x_j$ to $-x_j$) and $y_1$ to $iy_1$. Thus $\phi_1(P_j)=iP_j$, for all $1\leq j\leq 6$.
 Observe that $\phi_1^2=-\textrm{Id}$.

  \item[$ii)$] The automorphism $\psi_1$ of order 4 mapping $\z_5$ to $\z_5^2$. Observe that

$$P_1 \xmapsto{\psi_1} P_2  \xmapsto{\psi_1} P_3  \xmapsto{\psi_1} P_4  \xmapsto{\psi_1}  P_1,$$

\noindent as well as

$$iP_1 \xmapsto{\psi_1} iP_2 \xmapsto{\psi_1} iP_3 \xmapsto{\psi_1} iP_4 \xmapsto{\psi_1} iP_1.$$

\noindent The other $5$-torsion points are fixed by $\psi_1$.  

  \item[$iii)$] The automorphism $\rho_1$ of order 2 of the quadratic field of the complex multiplication, that
maps $i$ to $-i$. Observe that such an automorphism swaps $P_1$ and $P_3$ and swaps $P_2$ and
$P_4$

$$P_1 \overset{\rho_1}{\longleftrightarrow} P_3 \hspace{3cm} P_2 \overset{\rho_1}{\longleftrightarrow}  P_4.$$

\noindent Furthermore

$$iP_1 \overset{\rho_1}{\longleftrightarrow} -iP_3 \hspace{3cm} iP_2 \overset{\rho_1}{\longleftrightarrow}  -iP_4;$$

$$P_5 \overset{\rho_1}{\longleftrightarrow}  P_6 \hspace{3cm} iP_6 \overset{\rho_1}{\longleftrightarrow}  -iP_6.$$

\end{description}

\noindent By \cite[Chapter II, Theorem 2.3]{Sil2}, the extension $K_5/K(i)$ is abelian, thus
$\langle \phi_1, \psi_1 \rangle \simeq \ZZ/4\times \ZZ/4$, when all the conditions in the statement of Theorem \ref{gen2} hold. Moreover, with a quick computation,
one verifies that $\psi_1$ and $\rho_1$ commute. On the contrary $\phi_1$ and $\rho_1$  
do not commute in general, in fact

$$\rho_1\phi_1((x_1,y_1))=\rho_1((-x_1, iy_1))=(-x_3,-iy_3);$$ 

$$\phi_1\rho_1((x_1,y_1))=\phi_1((x_3, y_3))=(-x_3,iy_3).$$ 

\noindent Instead we have $\rho_1\phi_1((P_1))=\phi_1^{-1}\rho((P_1))$ and $\rho_1\phi_1((iP_1))=\phi_1^{-1}\rho_1((iP_1))$.
Being $\{P_1, iP_1\}$ a generating set for $K_5$, we can conclude $\rho_1\phi_1=\phi_1^{-1}\rho$.
Thus, when all the condition hold, we have $\langle\phi_1, \rho_1\rangle \simeq D_{8}$. We are going to describe the Galois groups
$G=\Gal(K_5/K)$, with respect to the degrees $[K_5:K]$.

 \begin{description}

\item[$d=32$]
\par If the degree $d$ of the extension $K_5/K$ is 32, then all the conditions hold. We have
$G=\langle \phi_1, \psi_1, \rho_1| \phi_1^4=\psi_1^4=\rho_1^2=\Id, \phi_1\psi_1=\psi_1\phi_1, \rho_1\psi_1=\psi_1\rho_1, \phi_1\rho_1=\phi_1^{-1}\rho_1 
\rangle\simeq D_{8} \times \ZZ/4\ZZ$.

\item[$d=16$]
If the degree $d$ of the extension $K_5/K$ is 16, then only one condition does not hold. 
\par If {\bf A} does not hold, then $\rho_1$ is the identity and we have an abelian group $G=\langle \phi_1, \psi_1 \rangle \simeq \ZZ/4\times \ZZ/4$.
\par  If one among  {\bf B1} and {\bf B2} does not hold, then $G\simeq D_8 \times \ZZ/2\ZZ$.
 \par  If one among  {\bf C} and {\bf D} does not hold, then $G\simeq \ZZ/4\ZZ \times (\ZZ/2\ZZ)^2$.

\item[$d=8$]
If the degree $d$ of the extension $K_5/K$ is 8, then two  conditions do not hold among the ones as above. 
\par If {\bf B1} and {\bf B2} do not hold, then $G\langle \phi_1, \rho_1\rangle \simeq D_8$. This is the only
case in which the Galois group $G$ is not abelian.
\par If one among  {\bf B1} and {\bf B2} does not hold and {\bf A} does not hold, then
$G\simeq \ZZ/4\ZZ \times \ZZ/2\ZZ$.
\par If one among  {\bf B1} and {\bf B2} does not hold and one among  {\bf C} and {\bf D} does not hold then
$G\simeq (\ZZ/2\ZZ)^3$.
\par If  one among  {\bf C} and {\bf D} does not hold and {\bf A} does not hold, then
$G\simeq \ZZ/4\ZZ \times \ZZ/2\ZZ$ again.

\item[$d=4$]
If  the degree $d$ of the extension $K_5/K$ is 4, then three conditions do not hold. If both {\bf B1} and {\bf B2} hold
or if  both {\bf C} and {\bf D} hold, then $G\simeq \ZZ/4\ZZ$. Otherwise $G\simeq \ZZ/2\ZZ \times \ZZ/2\ZZ$.

\item[$d\leq 2$]
If the degree $d$ of the extension $K_5/K$ is either 2 or 1, clearly the Galois group is respectively  $\Z/2\Z$ or $\{{\rm Id}\}$.

\end{description}

\section{Generators of  $K(\E[5])$ for elliptic curves $y^2=x^3+c$} \label{subgen1}

Let $\E_2: y^2=x^3+c$ be an elliptic curve of the family  $\mcF_2$.

\begin{remark} \label{rem1} Let $\phi_2$ denote the
complex multiplication of $\E_2$, i. e. $\phi_2(x,y)=(\z_3x,y)$. In many cases, if $P$ is a
nontrivial $m$-torsion point, then $\phi_2(P)$ is an $m$-torsion point that is not a multiple of $P$.
Therefore, in many cases a basis for $\E[m]$ is given by $\{P,\phi_2(P)\}$ and $K_m=K(x(P),y(P),\z_3)$.
Anyway, in a few special cases, the point $\phi_2(P)$ is a multiple of $P$ over the field $K(\z_3, \z_5)$. For example,
the abscissas of the $3$-torsion points of $\E_1$ are

$$\tilde{x_1}=0; \hspace{0.3cm} \tilde{x_2}=\sqrt[3]{-4c}; \hspace{0.3cm} \tilde{x_3}=\zeta_3\tilde{x_2}; \hspace{0.3cm} \tilde{x_4}=\zeta_3^2\tilde{x_2}.$$

\noindent Let $\tilde{P_h}$ be a point of abscissas $\tilde{x_h}$, for $1\leq h\leq 4$.
Clearly $\phi_2(\tilde{P_1})=\tilde{P_1}$ and then $\{\tilde{P_1},\phi_2(\tilde{P_1})\}$ is not a basis of $\E[3]$.
On the other hand,  $\{\tilde{P_h},\phi_2(\tilde{P_h})\}$ is a basis of $\E[3]$, for $2\leq h\leq 4$.
So we have to take care in our choice of $P$, when we use such a basis $\{P,\phi_2(P)\}$. 
For elliptic curves with complex multiplication $\phi_2$,
a generating set $\{x(P),y(P),\z_3\}$ is often easier to adopt than the one
in \eqref{genset}.
\end{remark}

\bigskip
 The abscissas of the points of order 5 of ${\mathcal{E}}$ are the roots of the polynomial

 $$r_5(x):=-5x^{12}-380cx^9+240c^2x^6+1600c^3x^3+256c^4.$$

\vspace{0.5cm}
   A factorization of $\varphi_1$ over $K(\zeta_5)$ is

$$r_5(x)= - 5\cdot (x^6  + (- 36\z_5^3  - 36\z_5^2  + 20)c x^3  +\frac{-288\z_5^3-288\z_5^2+176}{5} c^2 )$$

     $$\cdot (x^6  + ( 36\z_5^3  +36\z_5^2  + 56)c x^3  +\frac{288\z_5^3-288\z_5^2+464}{5} c^2 )$$

\bigskip\noindent and a factorization of $r_5(x)$ over $K(\z_3,\zeta_5)$ is

  $$r_5(x)= -5\cdot (x^3+\frac{(- 132\z_3 + 24)\z_5^3  + (36\z_3 + 108)\z_5^2  + (- 96\z_3 - 48)\z_5 - 48\z_3+116}{5}c)$$

\bigskip\hspace{1cm} $\cdot (x^3+$ {\Large $ \frac{( -36\z_3 -108)\z_5^3  + (-132\z_3 -156)\z_5^2  + (- 168\z_3 - 84)\z_5 - 84\z_3+8}{5}$ }$c)$

\bigskip\hspace{1cm} $\cdot (x^3+$ {\Large $\frac{( 132\z_3 + 156)\z_5^3  + (-36\z_3 + 72)\z_5^2  + ( 96\z_3 + 48)\z_5 + 48\z_3+164}{5}$} $c)$

\bigskip\hspace{1cm} $\cdot (x^3+$ {\Large $ \frac{( 36\z_3 -72)\z_5^3  + (132\z_3 -24)\z_5^2  + ( 168\z_3 + 84)\z_5 + 84\z_3+92}{5}$ } $c)$

\bigskip

\noindent Let \[
\begin{split} \delta_1 & :=-(\frac{(- 132\z_3 + 24)\z_5^3  + (36\z_3 + 108)\z_5^2  + (- 96\z_3 - 48)\z_5 - 48\z_3+116}{5}); \\
  \delta_2 & :=-(\frac{( -36\z_3 -108)\z_5^3  + (-132\z_3 -156)\z_5^2  + (- 168\z_3 - 84)\z_5 - 84\z_3+8}{5}); \\
 \delta_3 & :=-(\frac{(132\z_3 + 156)\z_5^3  + (-36\z_3 + 72)\z_5^2  + ( 96\z_3 + 48)\z_5 + 48\z_3+164}{5}); \\
  \delta_4 & :=-(\frac{( 36\z_3 -72)\z_5^3  + (132\z_3 -24)\z_5^2  + ( 168\z_3 + 84)\z_5 + 84\z_3+92}{5}). \\
 \end{split}
 \]

 \noindent Then the 12 roots of $r_5(x)$, i. e. the abscissas of the $5$-torsion points of $\E_1$,
are  $\sqrt[3]{\delta_1 c}$, $\sqrt[3]{\delta_1c}\z_3$,  $\sqrt[3]{\delta_1c}\z_3^2$,
$\sqrt[3]{\delta_2 c}$, $\sqrt[3]{\delta_2c}\z_3$,  $\sqrt[3]{\delta_2c}\z_3^2$,
$\sqrt[3]{\delta_3 c}$, $\sqrt[3]{\delta_3c}\z_3$,  $\sqrt[3]{\delta_3c}\z_3^2$,
$\sqrt[3]{\delta_4 c}$, $\sqrt[3]{\delta_4c}\z_3$,  $\sqrt[3]{\delta_4c}\z_3^2$.

\bigskip

\begin{thm} \label{zeta3} Let $\delta_1$ as above.
We have $K_5=K(\sqrt[3]{\delta_1c},\zeta_3,\sqrt{(\delta_1+1)c})$.
\end{thm}

\begin{proof} If $x_1:=\sqrt{\delta_1c}$, then by the factorization of $r_5(x)$ showed above,
we have that $x_1$ is the abscissas of a $5$-torsion point of $\E$. Let $y_1:=\sqrt{(\delta_1+1)c}$.
Then $P_1=(x_1,y_1)$ is a $5$-torsion point of $\E$.
By calculating $\phi_2(P_1)$ and the powers of $P_1$, one sees that $\phi_2(P_1)$ is not
a multiple of $P_1$. In fact $\phi_2(P_1)=(x_1\zeta_3,y_1)=(\sqrt{\delta_1c}\z_3,y_1)$ and
  $x(2P_1)=x(3P_1)=((\z_3 + 2)\z_5^3 + (-\z_3 + 1)\z_5^2 +1)\sqrt[3]{\delta_1c}$.
Thus
  $x(\phi_2(P_1))\neq x(nP_1)$, for all $1\leq n \leq 4$ (recall that $x(4P_1)=x(P_1)$).
  By Remark \ref{rem1}, then $\{P_1, \phi_1(P_1)\}$ form a basis of
$\E[5]$ and the conclusion is straightforward.
\end{proof}

Observe that  $[K_5:K]\leq 3\cdot 2\cdot 4\cdot 2=48$, for every $c\in K$.
This is in accordance with the fact that $\E$ has complex multiplication
$\phi_1: (x,y)\mapsto (\z_3x,y)$ and then the Galois representation

\[ \rho_{\E,5}:\Gal(\ov{K}/K) \rightarrow  \GL_2(\Z/5\Z) \]

\noindent is not surjective.




\section{Degrees $[K_5:K]$ for the curves of $\mcF_2$} \label{sub1}

As above, let $K$ be a number field and let $\E$ be an
elliptic curve defined over $K$.

\begin{thm}
Let $\E: y^2=x^3+c$, with $c\in K$. Let $\delta_1$ as above. Consider the conditions
\[
\begin{array}{lll}
& {\bf A.} \hspace{0.3cm} \z_3\notin K; \hspace{1.5cm}  & \\
& {\bf B1.} \hspace{0.3cm} \z_5+\z_5^{-1}\notin K(\z_3);  & {\bf C.} \hspace{0.3cm} \sqrt[3]{(\delta_1)c}\notin K(\z_3,\z_5); \\
& {\bf B2.} \hspace{0.3cm} \z_5\notin K(\z_3,\z_5+\z_5^{-1});  \hspace{1.5cm} & {\bf D.} \hspace{0.3cm} \sqrt{(\delta_1+1)c}\notin K(\z_3,\z_5).
\end{array}
\]

The possible degrees of the extension $K_5/K$ are the following

\begin{center}
\begin{tabular}{|c|c|c|c|}
\hline
 $d$ & {\em holding conditions}  &  $d$ & {\em holding conditions}  \\
\hline
 {\em 48} & {\bf A}, {\bf B1}, {\bf B2}, {\bf C}, {\bf D} &  {\em 6} & {\bf C} \textrm{and 1 among } {\bf A}, {\bf B1}, {\bf B2}, {\bf D} \\
\hline
 {\em 24} & {\bf C} \textrm{and 3 among } {\bf A}, {\bf B1}, {\bf B2}, {\bf D} &  {\em 4} & \textrm{2 among } {\bf A}, {\bf B1}, {\bf B2}, {\bf D} \\
\hline
 {\em 16} & {\bf A}, {\bf B1}, {\bf B2}, {\bf D} & {\em 3} &  {\bf C}  \\
\hline
 {\em 12} & {\bf C} \textrm{and 2 among } {\bf A}, {\bf B1}, {\bf B2}, {\bf D} &  {\em 2} & \textrm{1 among } {\bf A}, {\bf B1}, {\bf B2}, {\bf D}   \\
\hline
 {\em 8} & \textrm{  3 among } {\bf A}, {\bf B1}, {\bf B2}, {\bf D} &  {\em 1} & \textrm{ no holding conditions }   \\
\hline
\multicolumn{4}{c}{  }\\
\multicolumn{4}{c}{{\em Table 1}}
\end{tabular}
\end{center}

\end{thm}

\begin{proof}
Consider the tower of extensions
\[ K\subseteq K(\z_3)\subseteq K(\z_3,\z_5+\z_5^{-1}) \subseteq K(\z_3,\z_5) \subseteq K(\z_3,\z_5,\sqrt[3]{(\delta_1)c})
\subseteq K(\z_3,\z_5,\sqrt[3]{(\delta_1)c}, \sqrt{(\delta_1+1)c}).\]
The degree of $K_5/K$ is the product of the degrees of the intermediate extensions appearing in the tower.
Each of those extension gives a contribution to the degree that is less than or equal to 2, except the extension
$K(\zeta_3,\zeta_5,\sqrt[3]{(\delta_1)c})/K(\zeta_3,\zeta_5)$ that gives a contribution equal to
1 or 3. The final computation is straightforward.
\end{proof}

\section{Galois groups $\Gal(K_5/K)$ for the curves of $\mcF_2$} \label{sub2}

Let $\E_1$ be a curve of the family $\mcF_1$ and let $d:=|\textrm{Gal}(K(\E_1[5])/K)|$. To study the Galois group $G:=\textrm{Gal}(K(\E_1[5])/K)$,
we have to understand better the shapes of the coordinates of the $5$-torsion points of $\E_1$. Let $\delta_1, \delta_2,
\delta_3, \delta_4$ as in Section \ref{subgen1}. The 24 torsion points of $\E$ with exact order 5 are:

\[
\begin{array}{llllll}
 \pm P_1 &=(x_1,\pm y_1) &=\left(\sqrt[3]{\delta_1c}, \pm  \sqrt{(\delta_1+1)c}\right); \\
\pm \phi_2(P_1)  &=(\z_3x_1,\pm y_1) &=\left(\sqrt[3]{\delta_1c}\hspace{0.1cm} \z_3, \pm  \sqrt{(\delta_1+1)c}\right);  \\
\pm \phi_2^2(P_1) &=(\z_3^2x_1,\pm y_1) &= \left(\sqrt[3]{\delta_1c}\hspace{0.1cm} \z_3^2, \pm  \sqrt{(\delta_1+1)c}\right); \\
\end{array} \]

\[
\begin{array}{llllll}
  \pm P_2 &=(x_2,\pm y_2) &=\left(\sqrt[3]{\delta_2c}, \pm  \sqrt{(\delta_2+1)c}\right); \\
  \pm \phi_2(P_2)&=(\z_3x_2,\pm y_2) &=\left(\sqrt[3]{\delta_2c}\hspace{0.1cm} \z_3, \pm  \sqrt{(\delta_2+1)c}\right); \\
  \pm \phi_2^2(P_2) &=(\z_3^2x_2,\pm y_2) &= \left(\sqrt[3]{\delta_2c}\hspace{0.1cm} \z_3^2, \pm  \sqrt{(\delta_2+1)c}\right);\\
\end{array} \]

\[
\begin{array}{llllll}
 \pm P_3 &=(x_3,\pm y_3) &=\left(\sqrt[3]{\delta_3c}, \pm  \sqrt{(\delta_3+1)c}\right); \\
\pm \phi_2(P_3)  &=(\z_3x_3,\pm y_3) &=\left(\sqrt[3]{\delta_3c}\hspace{0.1cm} \z_3, \pm  \sqrt{(\delta_3+1)c}\right);  \\
\pm \phi_2^2(P_3) &=(\z_3^2x_3,\pm y_3) &= \left(\sqrt[3]{\delta_3c}\hspace{0.1cm} \z_3^2, \pm  \sqrt{(\delta_3+1)c}\right); \\
\end{array} \]

\[
\begin{array}{llllll}
  \pm P_4 &=(x_4,\pm y_4) &=\left(\sqrt[3]{\delta_4c}, \pm  \sqrt{(\delta_4+1)c}\right); \\
  \pm \phi_2(P_4) &=(\z_3x_4,\pm y_4) &=\left(\sqrt[3]{\delta_4c}\hspace{0.1cm} \z_3, \pm  \sqrt{(\delta_4+1)c}\right); \\
  \pm \phi_2^2(P_4) &=(\z_3^2x_4,\pm y_4) &= \left(\sqrt[3]{\delta_4c}\hspace{0.1cm} \z_3^2, \pm  \sqrt{(\delta_4+1)c}\right).\\
\end{array} \]

  \bigskip Thus we have the following four generating automorphisms of the
Galois group $G$.

\begin{description}
  \item[$i)$] The automorphism $\phi_2$ of the complex multiplication permuting the abscissas as follows

 $$\hspace{1cm}  \sqrt[3]{\delta_i c} \mapsto \sqrt[3]{\delta_i c}\z_3 \mapsto \sqrt[3]{\delta_i c}\z_3^2 \mapsto \sqrt[3]{\delta_i c},$$

\noindent for all $1\leq i\leq 4$, and fixing all the ordinates. Clearly $\phi_2$ has order 3.

  \item[$ii)$] The automorphism $\varphi_1$ of order 4  mapping $ \z_5$ to $\z_5^2,$, that consequentely maps

$$\delta_1  \mapsto \delta_2  \mapsto \delta_3  \mapsto \delta_4  \mapsto \delta_1 ,$$ 

\noindent i. e.

$$P_1 \xmapsto{\varphi_1} P_2 \xmapsto{\varphi_1} P_3  \xmapsto{\varphi_1} P_4 \xmapsto{\varphi_1}  P_1;$$

$$\phi_2{P_1} \xmapsto{\varphi_1} \phi_2{P_2}  \xmapsto{\varphi_1} \phi_2{P_3}  \xmapsto{\varphi_1} \phi_2{P_4}  
\xmapsto{\varphi_1}  \phi_2{P_1};$$

$$\phi_2^2{P_1} \xmapsto{\varphi_1} \phi_2^2{P_2}  \xmapsto{\varphi_1} \phi_2^2{P_3}  \xmapsto{\varphi_1} \phi_2^2{P_4}  
\xmapsto{\varphi_1}  \phi_2^2{P_1}.$$


  \item[$iii)$] The automorphism -Id of order 2, mapping
$\sqrt{(\delta_i+1)c}$ to $-\sqrt{(\delta_i+1)c}$, for all  $1\leq i\leq 4$, such that


$$P\xmapsto{-\textrm{Id}} -P,$$

\noindent for all $P\in \E[5]$.

 \item[$iv)$] The automorphism $\varphi_2$ of order 2 of the quadratic field of the complex multiplication mapping $\z_3$ to $\z_3^2,$ and
then swapping $\delta_1$ and $\delta_3$ and also
  $\delta_2$ and $\delta_4$. In particular
 
\[\begin{split}
P_1 & \overset{\varphi_2}{\longleftrightarrow} P_3 \hspace{4.4cm} P_2 \overset{\varphi_2}{\longleftrightarrow}  P_4;\\
& \\
\phi_2(P_1) & \overset{\varphi_2}{\longleftrightarrow} \phi_2^2(P_3) \hspace{3cm} \phi_2(P_2) \overset{\varphi_2}{\longleftrightarrow}  \phi_2^2(P_4);\\
& \\
\phi_2^2(P_1) & \overset{\varphi_2}{\longleftrightarrow} \phi_2(P_3) \hspace{3cm} \phi_2^2(P_2) \overset{\varphi_2}{\longleftrightarrow}  \phi_2(P_4).\\
\end{split}\]
\end{description}

One easily verifies that all these authomorphisms commute, except $\phi_2$ and $\varphi_2$.


Observe that $\psi_2:=\phi_2\circ \varphi_1$ form an homomorphism of order 12 and that
$G=\langle \psi_2, \varphi_2, -\textrm{Id}\rangle$. The automorphism $\psi_2$ and $\varphi_2$ does not commute,
since $\phi_2$ does not commute with $\varphi_2$,  instead one can verify that $\varphi_2\circ \psi_2=\psi_2^{-1} \circ \varphi_2$. 

Thus the group $\langle \psi_2, \varphi_2 \rangle$ has
a presentation $\langle \psi_2, \varphi_2 | \psi_2^{12}=\varphi_2^2=\Id, \varphi_2 \psi_2=\psi_2^{-1} \varphi_2\rangle\simeq D_{24}$.
If all the conditions  as in Table 1 hold, then we have a Galois group of order 48 $G=\langle \psi_2, \varphi_2 \rangle \times \langle -\Id \rangle\simeq
D_{24} \times \ZZ/2\ZZ$. By  \cite[Chapter II, Theorem 2.3]{Sil2}, the extension $K_5/K(\z_3)$ is abelian. Thus, if condition 
{\bf A} does not hold, then we have an abelian group. In all cases the group $G$ is
isomorphic to a subgroup of $D_{24} \times \ZZ/2\ZZ$ as follows.


\bigskip

\begin{description}

\item[$d=48$]
\par If the degree $d$ of the extension $K_5/K$ is 48, then all the conditions hold. We have
$G\simeq D_{24} \times \ZZ/2\ZZ$ as above.

\item[$d=24$]
If the degree $d$ of the extension $K_5/K$ is 24, then condition {\bf C} holds. \par If {\bf A} does not hold, then we have an abelian group. 
 In this case $G=\langle \psi_2, -\textrm{Id}\rangle \simeq \ZZ/12\ZZ \times \ZZ/2\ZZ$. 
\par  If {\bf A},  {\bf B1}, {\bf B2} and {\bf C} hold and {\bf D} does not hold, then
$G=\langle \varphi_2, \psi_2 \rangle\simeq D_{24}$.
\par If {\bf A}, {\bf C}  and {\bf D} hold and one among the conditions {\bf B1} and {\bf B2} does not hold, then
$G=\langle \varphi_2, \psi_2, -\Id\rangle$, where $\psi$ now has order $6$ and $\langle \varphi_2, \psi_2\rangle$ is
isomorphic to  $D_{12}$. We have $G\simeq D_{12}\times \langle \ZZ/2\ZZ \rangle$. 

 \item[$d=16$]
If  the degree $d$ of the extension $K_5/K$ is 16, then all the conditions hold but {\bf C}. Thus $\phi_2$ is the identity.
We have an abelian extension and an abelian Galois group $G=\langle\varphi_1,\varphi_2, -\Id \rangle\simeq \ZZ/4\ZZ \times (\ZZ/2\ZZ)^2$.

\item[$d=12$]
If the degree $d$ of the extension $K_5/K$ is 12, then condition {\bf C} holds. \par If {\bf A} does not hold, then we have an abelian group 
$G=\langle \psi_2, -\Id\rangle \simeq \ZZ/6\ZZ \times \ZZ/2\ZZ$. 
\par  If {\bf A} and {\bf C} hold, {\bf D} does not hold
and only one condition among {\bf B1} and {\bf B2} hold, then
$G=\langle \varphi_2, \psi_2 \rangle\simeq D_{12}$ (now $\psi_2$ has order 6).
\par If {\bf A}, {\bf C} and  {\bf D} hold, then 
$G\simeq D_{6}\times  \ZZ/2\ZZ  \simeq S_3\times \ZZ/2\ZZ $, where $S_3$ is the symmetric group of order 6.

\item[$d=8$]
If the degree $d$ of the extension $K_5/K$ is 8, then {\bf C} does not hold and we have again an abelian extension.
\par If {\bf D} does not hold, then $G\simeq \ZZ/4\ZZ \times \ZZ/2\ZZ$.
\par If  {\bf A} does not hold, then $G\simeq \ZZ/4\ZZ \times \ZZ/2\ZZ$.
\par If  one among {\bf B1} and {\bf B2} does not holds, then $G\simeq (\ZZ/2\ZZ)^3$.
 
\item[$d=6$]
If the degree $d$ of the extension $K_5/K$ is 6, then {\bf C} holds.
\par If  {\bf A} holds as well, then $G\simeq D_6\simeq S_3$.
\par If  {\bf A} does not hold, then $G\simeq \ZZ/6\ZZ$.

\item[$d=4$]
If  the degree $d$ of the extension $K_5/K$ is 4, then {\bf C} does holds. If both  {\bf B1} and {\bf B2} hold, then  $G\simeq \ZZ/4\ZZ$,
otherwise $G$ is isomorphic to the Klein group $\ZZ/2\ZZ\times \ZZ/2\ZZ$. 

\item[$d\leq 3$]
If the degree $d$ of the extension $K_5/K$ is 3 or 2 or 1, clearly the Galois group is respectively $\Z/3\Z$,  $\Z/2\Z$ or $\{{\rm Id}\}$.

\end{description}

 \section{Some applications} \label{lastsub} \normalcolor

We are going to show some applications of the results achieved in the previous sections.
In particular we will show an application to the local-global divisibility problem,
an immediate application to modular curves and another one to Shimura curves.

\subsection{A minimal bound for the local-global divisibility by $5$}

We recall the statement of the Local-Global Divisibility Problem and some key facts about the cohomology group that gives the
obstruction to its validity in order to maintain the paper more self-contained. For further details one can see \cite{DZ1}, \cite{DP} and \cite{PRV2}. 

\begin{Problem}[Dvornicich, Zannier, 2001] \label{prob1}
Let $K$ be a number field, $M_K$ the set of the places $v$ of $K$ and $K_v$ the completion of $K$ at $v$.
Let $\G$ be a commutative algebraic group defined over $k$. Fix a positive integer $m$ 
and assume that there exists a $K$-rational point $P$ in $\G$, such that  $P=mD_v$, for some $D_v\in {\mathcal{G}}(K_v)$,
for all but finitely many  $v\in M_K$. Does there exist $D\in \G(K)$ such that $P=mD$?
\end{Problem}

\bigskip\noindent  The classical question is considered for all commutative algebraic groups, but in our situation,
we can confine the discussion only to elliptic curves $\E$ over $K$.
Let $P\in \E[m]$ and let $D\in \E(\bar{K})$ be a $m$-divisor of
$P$, i. e. $P=mD$. For every $\sigma\in G=\Gal(K_5/K)$, we have

$$m\sigma(D)=\sigma(mD)=\sigma(P)=P.$$

\noindent Thus $\sigma(D)$ and $D$ differ by a point in $\E[m]$ 
and we can construct a cocycle $\{Z_{\sigma}\}_{\sigma\in G}$ of $G$ with
values in $\E[m]$ by

\begin{equation} \label{eq1} Z_{\sigma}:=\sigma(D)-D. \end{equation} 

\noindent Such a cocycle vanishes in $H^1(G,\E[m])$, if and only if there exists a $K$-rational
$m$-divisor of $P$ (see for example \cite{DZ1} or \cite{DP}). In particular, the hypotheses
about the validity of the local-dividibility in Problem \ref{prob1} imply that
the cocyle $\{ Z_{\sigma}\}_{\sigma\in G}$ vanishes in  $H^1(\Gal((K_m)_v/K_v),\E[m])$,
for all but finitely many $v\in M_K$. Let $G_v$ denote the group $\Gal((K_m)_v/K_v)$
and let $\Sigma$ be the subset of $M_K$ containing all the $v\in M_K$, that are
unramified in $K_m$. Then $G_v$ is a cyclic subgroup of $G$, for all $v\in \Sigma$.
Moreover, in \cite{DZ1} Dvornicich and Zannier observe that
by the Chebotarev Density Theorem, the local Galois group $G_v$ varies over \emph{all} cyclic subgroups of
$G$ as $v$ varies in $\Sigma$. Thus, they state the following definition about a subgroup of
$H^1(G,\E[m])$ which portrays the hypotheses of the
problem in this cohomological context and essentially 
gives the obstruction to the validity of such Hasse principle (see also \cite{DZ3}).

\begin{definition} 
 A cocycle $\{Z_{\sigma}\}_{\sigma\in G}\in H^1(G,\E[m])$ satisfies the
\emph{local conditions} if, for every $\sigma\in G$, there exists $A_{\sigma}\in \E[m]$ such that
$Z_{\sigma}=(\sigma-1)A_{\sigma}$. The subgroup of $H^1(G,\E[m])$ formed by all the cocycles satisfying the local conditions
is the first local cohomology group $H^1_{\loc}(G,\E[m])$.
\end{definition}

\noindent Thus

\begin{equation} \label{h1loc}
H^1_{\textrm{\loc}}(G,\E[m])=\bigcap_{v\in \Sigma} (\ker  H^1(G,\E[m])\xrightarrow{\makebox[1cm]{{\small $res_v$}}} H^1(G_v,\E[m])).
\end{equation}

\noindent The triviality of $H^1_{\textrm{\loc}}(G,\E[m])$ assures the validity of the local-global divisibility by $m$ 
in $\E$ over $K$.
 
 \begin{thm}[Dvornicich, Zannier, 2001] \label{h1loc2} 
If $H^1_{\textrm{\loc}}(G,\E[m])=0$, then the local-global divisibility by $m$ holds in $\E$ over $k$.
\end{thm}

\noindent In \cite{DZ1} the authors showed that the local-global divisibility by $5$ holds in $\E$ over $k$,
for all $l\geq 1$ (see also \cite{Won}). Anyway in that paper, as well as in all the other papers (of various authors) about the topic, there
is no information about the minimal number of places $v$ for which the validity of the local
divisibility by a prime $p$ in $\E$ over $K_v$ is sufficient to have the global divisibility by $p$ in $\E$ over $K$.

\bigskip\noindent For the first time, here we show such a lower bound  for the number of places $v$ 
 when $p=5$, in the case of the curves belonging to the families $\mcF_1$ and $\mcF_2$.

By Theorem \ref{h1loc2}, the triviality of the first cohomology group $H^1_{\loc}(G, \E[m])$
is a sufficient condition to have an affirmative answer to Problem \ref{prob1}.  We have already
recalled that by the Tchebotarev Density Theorem, the group $G_v$ varies over all
the cyclic subgroups of $G$, as $v$ varies among all the places of $K$, that are unramified in $K_m$.
Observe that in fact we have

$$H^1_{\loc}(G,\E[m])=\bigcap_{v\in S} (\ker  H^1(G,\E[m])\xrightarrow{\makebox[1cm]{{\small $res_v$}}} H^1(G_v,\E[m])),$$

\noindent where $S$ is a subset of $\Sigma$ such that,
for all $v, w\in S$, with $v\neq w$, the groups $G_v$ and $G_w$ correspond to distinct cyclic subgroups of $G$. 
If we are able to find such
an $S$ and to prove that the local-global divisibility by $5$ holds in $\E(K_v)$, for all $v\in S$, then
we get $H^1_{\loc}(G,\E[m])=0$ (and consequently the validity of the Hasse principle for
disivibility by $5$ in $\E$ over $K$). Observe that in particular the set $S$ is finite (on the contrary $\Sigma$ is not finite). So it suffices to have the
local-global divisibility by $5$ for a finite number of suitable places to get the global divisibility by $5$.

\par In view of the results achieved for the Galois groups
$\Gal(K_5/K)$ for elliptic curves of the families $\mcF_1$ and $\mcF_2$, we can prove
that $S$ could be chosen as a subset of $S$ with a cardinality surprisingly small.

\begin{thm} \label{appl_1}
Let $m$ be a positive integer. Let $\E$ be an elliptic curves defined over a number field $K$, with Weierstrass equation $y^2=x^2+bx$, for some $b\in K$.
Let $S$ be a subset of $M_K$ of places $v$ unramified in $K_m$, with cardinality $|S|=7$, such that 
$G_v$ varies among all the cyclic subgroups of $G$, as $v$ varies in $S$.
Let $P\in \E(K)$ such that $P=mD_v$, for some $D_v\in \E(K_v)$, for all $v\in S$. Then there exists $D\in \E(K)$ such that
$P=mD$.
\end{thm}

\begin{proof}
Let $s$ be the number of distinct cyclic subgroups of $G$. Since the group $G_v$ varies over all the cyclic subgroups of $G$, 
as $v$ varies in $M_k$, we can choose $S$ as a subset of $M_k$ with cardinality $s$, such that $G_v$ and $G_w$ are
pairwise distinct cyclic subgroups of $G$, for all $v, w\in S$, $v\neq w$. We have just to show that $s=7$.
We have proved in  Section \ref{gal1}, that for every $\E\in \mcF_1$, the Galois group $G$ is
isomorphic to a subgroup of $D_8\times \ZZ/4\ZZ$. The group $D_8$ has 5 cyclic subgroups, namely
$\langle \phi_1\rangle \simeq \ZZ/4\ZZ$, $\langle \phi_1^2\rangle\simeq \ZZ/2\ZZ$, $\langle \rho\rangle\simeq \ZZ/2\ZZ$,
 $\langle \phi_1\rho\rangle \simeq \ZZ/2\ZZ$, $\langle \phi_1^2\rho\rangle\simeq \ZZ/2\ZZ$. In addition we have
the cyclic subgroups $\langle \psi_1\rangle \simeq \ZZ/4\ZZ$ and $\langle \psi_1^2\rangle\simeq \ZZ/2\ZZ$.
Thus $G$ has at most 7 cyclic subgroups and we get the conclusion.  
\end{proof}

\begin{thm}
Let $m$ be a positive integer. Let $\E$ be an elliptic curves defined over a number field $K$, with Weierstrass equation $y^2=x^2+c$, for some $c\in K$.
Let $S$ be a subset of $M_K$ of places $v$ unramified in $K_m$, with cardinality $|S|=13$, such that 
$G_v$ varies among all the cyclic subgroups of $G$, as $v$ varies in $S$.
Let $P\in \E(K)$ such that $P=mD_v$, for some $D_v\in \E(K_v)$, for all $v\in S$. Then there exists $D\in \E(K)$ such that
$P=mD$.
\end{thm}

\begin{proof}
Let $s$ be the number of distinct cyclic subgroups of $G$. As in the proof of Theorem \ref{appl_1} 
we can choose $S$ as a subset with cardinality $s$, such that $G_v$ and $G_w$ are
pairwise distinct cyclic subgroups of $G$, for all $v, w\in S$, with $v\neq w$. We have just to show that $s=13$.
We have proved in  Section \ref{sub2}, that for every $\E\in \mcF_2$, the Galois group $G$ is
isomorphic to a subgroup of $D_{24}\times \ZZ/2\ZZ$. The group $D_{24}$ has 11  cyclic subgroups, namely
$\langle \psi_1\rangle \simeq \ZZ/12\ZZ$, $\langle \psi_1^2\rangle\simeq \ZZ/6\ZZ$, $\langle \psi_1^3\rangle\simeq \ZZ/4\ZZ$, 
$\langle \psi_1^4\rangle\simeq \ZZ/3\ZZ$, $\langle \psi_1^6\rangle\simeq \ZZ/2\ZZ$, 
$\langle \varphi_2\rangle\simeq \ZZ/2\ZZ$,
$\langle \psi_1 \varphi_2\rangle \simeq \ZZ/12\ZZ$, $\langle \psi_1^2 \varphi_2\rangle\simeq \ZZ/6\ZZ$, $\langle \psi_1^3 \varphi_2\rangle\simeq \ZZ/4\ZZ$, 
$\langle \psi_1^4 \varphi_2\rangle\simeq \ZZ/3\ZZ$, $\langle \psi_1^6 \varphi_2\rangle\simeq \ZZ/2\ZZ$.
In addition, we have the cyclic subgroups $\langle -\Id\rangle \simeq \ZZ/2\ZZ$ and
$\langle \psi_1^4, \rangle\times \langle -\Id\rangle\simeq \ZZ/6\ZZ$. 
Thus $G$ has at most 13 cyclic subgroups and we get the conclusion.  
\end{proof}

 \subsection{Remarks on modular curves}

We recall some basic definitions about modular curves; for further details one can see for instance \cite{KM} and \cite{Shi}. As usual, we denote by $\mathcal{H}=\{z\in \mathbb{C}\,:\,Im\,z > 0\}$ the complex upper half plane. It is well-known that the group  $\SL_2(\Z)$ acts
 on $\mathcal{H}$ via the M\"obius trasformations

\[ \left( \begin{array}{cc} a & b \\ c & d \end{array} \right)z= \frac{az+b}{cz+d} \ . \]

\noindent For every positive integer $N$,  the {\it principal congruence group of level $N$} is the set

\[ \Gamma(N) = \left\{ A=\left( \begin{array}{cc} a & b \\ c & d \end{array} \right) \in \SL_2(\Z)\,\mid \,
A\equiv \left( \begin{array}{cc} 1 & 0 \\ 0 & 1 \end{array} \right) \pmod N \right\}. \ \]

\noindent   A {\em congruence group} is a subgroup $\Gamma$ of $\SL_2(\Z)$ containing $\Gamma(N)$, for some  $N$. When $N$ is minimal,
  the congruence group is said to be {\it of level $N$}.  For every $N$, the most important congruence groups of level $N$ are $\Gamma(N)$ itself and the groups:

\[ \Gamma_1(N) = \left\{ A=\left( \begin{array}{cc} a & b \\ c & d \end{array} \right) \in \SL_2(\Z)\,\mid \,
A\equiv \left( \begin{array}{cc} 1 & * \\ 0 & 1 \end{array} \right) \pmod N \right\} \ ,\]
and
\[ \Gamma_0(N) = \left\{ A=\left( \begin{array}{cc} a & b \\ c & d \end{array} \right) \in \SL_2(\Z)\,\mid \,
A\equiv \left( \begin{array}{cc} * & * \\ 0 & * \end{array} \right) \pmod N \right\} \ .\]

\noindent The quotient  $\mathcal{H}/\Gamma$ of $\mathcal{H}$ 
by the action of $\Gamma$, equipped with the analytic structure induced by $\mathcal{H}$, is a Riemann surface, 
that is denoted by $Y_\Gamma$.  The modular curve $X_\Gamma\,$, associated to $\Gamma$, is the
compactification of $Y_\Gamma$  by the addition of a finite set of
rational points corresponding to the orbits of $\mathbb{P}^1(\QQ)$ under $\Gamma$, i. e. the {\it cusps}. \\
\noindent The modular curves associated to the groups $\Gamma_0(N)$ and $\Gamma_1(N)$
are denoted respectively by $X_0(N)$ and $X_1(N)$.  The modular curve associated to $\Gamma(N)$ is denoted by $X(N)$.

\bigskip

The curves $X(N)$, $X_1(N)$ and $X_0(N)$ are spaces of moduli of families of elliptic curves with an extra structure of level $N$ as
follows  (for further details see for example \cite{KM}, \cite{Kn} and \cite{Shi}).

\bigskip
\noindent \begin{thm}  Let $N$ be a positive integer and let $X(N)$, $X_1(N)$ and $X_0(N)$ as above. Then

\begin{description}

\item[ i)] 
 non cuspidal points in $X(N)$ correspond to triples $(\E,P,Q)$, where $\E$ is an elliptic curve
(defined over $\mathbb{C}$) and $P$, $Q$ are points of order $N$ generating $\E[N]$;

\item[ ii)]  non cuspidal points in $X_1(N)$ correspond to pairs $(\E,P)$, where $\E$ is an elliptic curve
(defined over $\mathbb{C}$) and $P$ is a point of order $N$; 

\item[ iii)] non cuspidal points in $X_0(N)$ correspond to couples $(\E,C_N)$, where $\E$ is an elliptic curve
(defined over $\mathbb{C}$) and $C_N$ is a cyclic subgroup of $\E[N]$ of order $N$. 

\end{description}

\end{thm}

\bigskip\noindent A point on a modular curve, which corresponds to an elliptic curve with
complex multiplication  is called a \emph{CM-point}.

 We can deduce the following facts from what showed in the previous sections 
(see in particular Theorem \ref{zeta3} and Theorem \ref{gen2}).

\begin{prop}
Let $K$ be a number field. Let $\E_1\in \mcF_1$ and let $P\in \E_1[5]$
such that $\{P,\phi_1(P)\}$ is a basis of $\E_1[5]$ (as above $\phi_1$ denotes the complex
multiplication of $\E_1$).
Then

\[
\begin{split}
& \textrm{the pair $(\E_1,\langle P \rangle)$ defines a non-cuspidal  $K$-rational CM-point of $X_0(5)$}, \\
& \textrm{if and only if $(\E_1,P)$ defines a non-cuspidal $K$-rational CM-point of $X_1(5)$, }\\
& \textrm{if and only if $(\E_1,P, \phi_1(P))$ defines a $K$-rational CM-point of $X(5)$. }\\
\end{split}
\]
\end{prop}

\bigskip

\begin{prop}
Let $K$ be a number field. Let $\E_2\in \mcF_2$ and let $P\in \E_2[5]$
such that $\{P,\phi_12P)\}$ is a basis of $\E_2[5]$ (as above $\phi_2$ denotes the complex
multiplication of $\E_2$).
Then

\[
\begin{split}
& \textrm{the pair $(\E_2,\langle P \rangle)$ defines a non-cuspidal  $K$-rational CM-point of $X_0(5)$}, \\
& \textrm{if and only if $(\E_2,P)$ defines a non-cuspidal $K$-rational CM-point of $X_2(5)$, }\\
& \textrm{if and only if $(\E_2,P, \phi_2(P))$ defines a $K$-rational CM-point of $X(5)$. }\\
\end{split}
\]
\end{prop}

 \subsection{Remarks on Shimura curves} \label{sec_shimura}
We are going to describe two curves, the Shimura curves named ${\mathcal{X}}_0(N)$ and $\X_1(N)$, which are generalizations of
the modular curves $X_0(N)$ and $X_1(N)$, i. e. they are moduli spaces of certain abelian varieties of dimension 2,
with some $N$-level structures.

\bigskip
We firstly recall that a central $K$-algebra is an algebra over $K$ with center $K$. Furthermore a simple
$K$-algebra is an algebra over $K$ with nontrivial two-sided ideals. A division $K$-algebra 
is an algebra ${\mathfrak{A}}$ over the field $K$, in which for every $a_1,a_2\in {\mathfrak{A}}$, with $a_2\neq 0$, there
exists $b\in {\mathfrak{A}}$ such that $a_1=a_2b$.

\begin{definition}

A quaternion algebra over $K$ is a central simple algebra over $K$ of dimension 4.
\end{definition}

\noindent By Wedderburn's Theorem (see \cite{Wed}), every central simple $K$-algebra is a matrix algebra over a central division $K$-algebra.
The division central $K$-algebra are classified by the Brauer group $\textrm{Br}(K)=H^2(\Gal(\bar{K}/K), \bar{K})$.

\bigskip\noindent 
One of the simplest example of a quaternion $K$-algebra is 
the set $M_2(K)$  of $2\times 2$ matrices with entries in $K$.
The quaternion algebra $\Bi=M_2(\QQ)$ over $\QQ$ is important in the definition of Shimura curves.

\begin{definition}
Let $R$ be the ring of integers of $K$. An order 
of a quaternion $K$-algebra  ${\mathfrak{A}}$  is a $R$-lattice, which is also a subring.
A maximal order 
of a quaternion $K$-algebra  ${\mathfrak{A}}$ is an order that is not contained in
any other order, i. e. 
it is a $R$-lattice of rank 4, which is also a subring.
An Eichler order is given by the intersection of two maximal orders. 
 
\end{definition}


\noindent For example, the set $M_2(R)$ of $2\times 2$ matrices with entries in $R$ is a maximal order of $M_2(K)$.
In particular $M_2(\ZZ)$ is a maximal order of $M_2(\QQ)$. Let $N$ be a positive integer. 
Observe that  the subgroup  $\Or'_N$ of $M_2(\QQ)$, formed by the matrices

$$  \left( \begin{array}{cc} a & Nb \\ N^{-1}c & d \end{array} \right)    $$

\noindent is conjugate to $M_2(\ZZ)$ by 

$$\left( \begin{array}{cc} N & 0 \\ 0 & 1 \end{array} \right).$$ 

\noindent Thus $\Or'_N$ is a maximal order too and $\Or_N:=M_2(\ZZ)\cap \Or'_N$ is an Eichler order. We have
that $\Or_N$ is the subset of  $M_2(\ZZ)$, formed by the  matrices of the form

$$\left( \begin{array}{cc} a & b \\ Nc & d \end{array} \right).$$

\bigskip\noindent Observe that $\Or_N$ is an analogous in $M_2(\ZZ)$ of $\Gamma_0(N)$ in $\SL_2(\ZZ)$ in the case of modular curves.

\begin{definition} For every order $\Or$, we denote by $\Or^{1}$ the subset of its, formed by the elements
of norm 1. \end{definition}

\begin{definition}
The quotient $\X(\Or):=\Hil/\Or^1$ is called a Shimura curve. In particular we denote by $\X(1)$ the Shimura
curve obtained by the order $\Or=M_2(\ZZ)$ of $\Bi$. 
\end{definition}

\noindent In the literature, the curve $\X(1)$ is also denoted by ${\mathcal{M}}^{\Bi}$ or simply by ${\mathcal{M}}$. 
It turns out that the Shimura curves are connected and compact. So we do not need to
add cusps as in the case of modular curves. Indeed, we have such a moduli interpretation
of the curve $\X(1)$ (see for example \cite{Cla} and see also \cite{Shi}).

\bigskip \noindent 

\begin{thm} \label{Shimura}
The curve $\X(1)$ is the moduli space of the couples $(\Ab,\iota)$, where $\Ab$ is an abelian 
surface principally polarized such that either $\A$ is simple or $\A=\E\times \E$, 
where $\E$ is a CM elliptic curve defined over $K$, 
and $\iota : \Or \rightarrow {\rm End}(\Ab)$ is an embedding. 
\end{thm}

\bigskip\noindent When an abelian surface $\Ab$ admits an embedding $\iota : \Or \rightarrow {\rm End}(\Ab)$, one
says that $\Ab$ has a quaternion multiplication or that $\Ab$ is a $QM$-abelian
surface. Sometimes $\A$ is also said a $\Or-QM$-abelian surface, to underline
that the quaternion multiplication is given by the order $\Or$.

\bigskip\noindent By Theorem \ref{Shimura}, it is clear the strong relation between
CM elliptic curves and Shimura curves. The points on Shimura curves parametrizing squares of CM elliptic curves are often called
CM points. 
\bigskip \par For a positive integer $N$, the Shimura curve $\X(\Or_N):=\Hil/\Or_N^1$ is similar to the Shimura curve $\X(1)$, but with an extra structure
of level $N$. Because of the connection between $\Or_N$ and $\Gamma_0(N)$, the curve $\X(\Or_N)$ is often
denoted by $\X_0(N)$ (and sometimes by $M_0^{\Bi}(N)$, as for instance in \cite{Arai}). We recall the moduli interpretation for $\X_0(N)$ 
(see again \cite{Cla} and also \cite{Shi}).

\bigskip

\begin{thm} \label{Shimura2}
Let $\Or=M_2(\ZZ)$. The curve $\X_0(N)$ is the moduli space of the triples $(\Ab,\iota,W)$, where $\Ab$ is a $QM$-abelian 
surface principally polarized such that either $\A$ is simple or $\A=\E\times \E$, 
where $\E$ is a CM elliptic curve defined over $K$, the quaternionic multiplication is given by
$\iota : \Or \rightarrow {\rm End}(\Ab)$ and $W\in \Ab[N]$  is a cyclic $\Or$-submodule of order $N^2$, which is isomorphic to $ (\ZZ/N\ZZ)^2$ as an abelian group.
\end{thm}

\bigskip\noindent In a similar way as for $\Or_N$, one can take the Eichler order 

$$\Or_{1,N}:=\left\{ A=\left( \begin{array}{cc} a & b \\ c & d \end{array} \right) \in M_2(\Z) \mid
A\equiv \left( \begin{array}{cc} 1 & * \\ 0 & 1 \end{array} \right) \pmod N \right\},$$

\noindent  which is an analogous of $\Gamma_1(N)$. In this case we get the Shimura curve $\X(\Or_{1,N})$,
that is denoted by $\X_1(N)$   (and sometimes by $M_1^{\Bi}(N)$). Let $e_1, e_2$ denote two standard idempotents for
$M_2(\ZZ/N\ZZ)$. The moduli interpretation for $\X_1(N)$ is the following  (see \cite{Cla} and see also \cite{Shi}).

\begin{thm} \label{Shimura3}
 Let $\Or=M_2(\ZZ)$. The curve $\X_1(N)$ is the moduli space of the $4$-tuples $(\Ab,\iota,W,P)$, where $\Ab$ is a $QM$-abelian 
surface principally polarized such that either $\A$ is simple or $\A=\E\times \E$, 
where $\E$ is a CM elliptic curve defined over $K$, the quaternionic multiplication is given by
$\iota : \Or \rightarrow {\rm End}(\Ab)$, $W\in \Ab[N]$ is a stable $\Or$-module and $P=e_1W$ is a point of order $N$. 

\end{thm}

\bigskip\noindent
Observe that if $(\Ab,\iota,W,P)$ corresponds to a CM point of $\X_1(N)$, then $\Ab=\E\times \E$ (for some CM elliptic curve) and
 $W=P\times Q\in \E^2[N]$, with $P=e_1W\in \E[N]$ and $Q\in \E[N]$.

\bigskip
For some Shimura varieties and certain fields $F$ it is known that the set of $F$-rational points is empty (see for instance \cite{RVP} and \cite{CX}). 
For $j\in \{0,1\}$, let $\X_j(5)(K)$ denote the sets of the $K$-rational points of the Shimura curve
$\X_j(5)$. Let $\theta_1$ as in Section \ref{sec_gen_1} and $\delta_1$ as in Section \ref{subgen1}.
We have that the sets $\X_0(5)(K)$ and  $\X_1(5)(K)$ are nonempty whenever $K=K_5$ is one of the fields
$\QQ( \zeta_5, i,\sqrt{(\theta_1+1)b\sqrt{\theta_1b}})$ and $\QQ (\sqrt[3]{\delta_1c},\zeta_3,\sqrt{(\delta_1+1)c})$.

\bigskip\noindent More precisely, by the results achieved about the fields $K_5$ for elliptic curves of the families $\mcF_1$ and $\mcF_2$, 
we can make the following remarks about the points of the curves $\X_0(5)$ and  $\X_1(5)$.

\begin{prop} Let $K$ be a number field, let $\E_1\in \mcF_1$ and let $P_1=\left(\sqrt{\theta_1b}, \pm \sqrt{(\theta_1+1)b\sqrt{\theta_1b}}\right)$.
In particular $(P_1, \phi_1(P_1))\in \E_1\times \E_1$.
Let $\Or=M_2(\ZZ)$ and let $\iota:\Or \rightarrow {\rm End}(\E_1\times \E_1)$.
 Then
\begin{description}
\item[i)] the triple ($\E_1\times \E_1$, $\iota$, $(P_1 ,\phi_1(P_1))$  
corresponds to a  CM-point of $\X_0(5)$; 
\item[ii)] the $4$-tuple
 ($\E_1\times \E_1$, $\iota$, $(P_1 ,\phi_1(P_1))$, $P_1$) 
corresponds to a  CM-point of $\X_1(5)$.
\end{description}
\end{prop}

\bigskip

\begin{prop} Let $K$ be a number field, let $\E_2\in \mcF_2$ and let $P_1=\left(\sqrt[3]{\delta_1c}, \pm  \sqrt{(\delta_1+1)c}\right)$.
In particular $(P_1, \phi_2(P_1))\in {\rm End}\E_2\times \E_2$. Let $\Or=M_2(\ZZ)$ and let $\iota:\Or \rightarrow {\rm End}(\E_2\times \E_2)$.
  Then
\begin{description}
\item[i)] the triple ($\E_2\times \E_2$, $\iota$, $(P_1 ,\phi_1(P_1))$  
corresponds to a  CM-point of $\X_0(5)$;
\item[ii)] the $4$-tuple
 ($\E_1\times \E_1$, $\iota$, $(P_1 ,\phi_1(P_1))$, $P_1$)  
corresponds to a  CM-point of $\X_1(5)$.
\end{description}
\end{prop}



\vskip 0.7cm \noindent 
\textbf{Acknowledgments.} I would like to thank Andrea Bandini for useful discussions and for
some precious remarks about a preliminary version of this paper. I produced part of this work when
I was a guest at the Max Planck Institute for Mathematics in Bonn. I am grateful to all people there
for their kind hospitality and for the excellent work conditions.

\vskip 0.5cm

Laura Paladino\par\smallskip
University of Calabria \par
Ponte Bucci, Cubo 30B\par
87036 Arcavacata di Rende\par
Italy\par 
e-mail address: paladino@mat.unical.it

\end{document}